\documentclass[11pt]{amsart}
\usepackage{amsmath,amsthm,amsfonts,amssymb,amstext}
\theoremstyle{plain}

\usepackage{amssymb}
\usepackage{amscd}

\newtheorem{theorem}{Theorem}[section]

\newtheorem{lemma}{Lemma}[section]

\newtheorem{prop}{Proposition}[section]
\newtheorem{definition}{Definition}[section]

\theoremstyle{remark}
\newtheorem{remark}[theorem]{Remark}

\def\A{\operatorname{A}}
\def\B{\operatorname{B}}
\def\C{\operatorname{C}}
\def\D{\operatorname{D}}

\def\I{\operatorname{I}}
\def\J{\operatorname{J}}

\def\O{\operatorname{O}}
\def\U{\operatorname{U}}
\def\M{\operatorname{M}}
\def\Z{\operatorname{Z}}

\def\PU{\operatorname{PU}}
\def\SO{\operatorname{SO}}
\def\Sp{\operatorname{Sp}}
\def\SU{\operatorname{SU}}
\def\PO{\operatorname{PO}}
\def\PSO{\operatorname{PSO}}
\def\PSp{\operatorname{PSp}}
\def\PSU{\operatorname{PSU}}

\def\Ad{\operatorname{Ad}}

\def\Aut{\operatorname{Aut}}

\def\diag{\operatorname{diag}}
\def\defe{\operatorname{defe}}

\def\det{\operatorname{det}}
\def\dim{\operatorname{dim}}

\def\Hom{\operatorname{Hom}}

\def\Im{\operatorname{Im}}

\def\Int{\operatorname{Int}}

\def\rank{\operatorname{rank}}
\def\sign{\operatorname{sign}}

\def\tr{\operatorname{tr}}

\newcommand{\frg}{\mathfrak{g}}

\newcommand{\fru}{\mathfrak{u}}

\newcommand{\bbC}{\mathbb{C}}

\newcommand{\bbR}{\mathbb{R}}
\newcommand{\bbZ}{\mathbb{Z}}
\newcommand{\bbH}{\mathbb{H}}

\begin{document}

\title{Maximal abelian subgroups of compact matrix groups}
\author{Jun Yu}
\date{}

\abstract{We classify closed abelian subgroups of the automorphism group of any compact classical simple
Lie algebra whose centralizer has the same dimension as the dimension of the subgroup, and describe Weyl
groups of maximal abelian subgroups.}
\endabstract

\maketitle



\tableofcontents

\section{Introduction}

In this paper, we study closed abelian subgroups $F$ of a compact simple Lie group $G$ satisfying the condition
of \[\dim\frg_0^{F}=\dim F,\tag{$\ast$}\] where $\frg_0$ is the Lie algebra of $G$. It is clear that any maximal
abelian subgroup of $G$ satisfies the condition $(\ast)$. A good property of this class of abelian subgroups is
the following: given a surjective homomorphism $p: G_1\longrightarrow G_2$ and a closed abelian subgroup $F$ of
$G_1$, $F$ satisfies the condition $(\ast)$ if and only if $p(F)$ satisfies the condition $(\ast)$. Precisely, we
classify closed abelian subgroups of the automorphism group $G=\Aut(\fru_0)$ satisfying the condition $(\ast)$ for
$\fru_0$ a compact classical simple Lie algebra (except $\mathfrak{so}(8)$). In other publications, we classify
closed abelian subgroups satisfying the condition $(\ast)$ of the automorphism group of any other compact simple
Lie algebra.


The method of this paper is through linear algebra. We have four cases to consider: subgroups of the projective
unitary group $\PU(n)$; of the projective orthogonal group $\O(n)/\langle-I\rangle$; of the projective symplectic
group $\Sp(n)/\langle-I\rangle$; and subgroups of $\PU(n)\rtimes\langle\tau_0\rangle$
($\tau_0=\textrm{complex conjugation}$) not contained in $\PU(n)$. The method of classification for abelian
subgroups of $\PU(n)$ is as follows. Given a closed abelian subgroup $F$ of $\PU(n)$, we define an antisymmetric
bimultiplicative function \[m: F\times F\longrightarrow\U(1)=\{z\in\bbC:|z|=1\}\] by $m(x,y)=\lambda$ for any
$x=[A],y=[B]\in F$, $A,B\in\U(n)$ with $ABA^{-1}B^{-1}=\lambda I$. Let $\ker m=\{x\in F:m(x,y)=1,\forall y\in F\}$,
which is a subgroup of $F$. By linear algebra, $(F/\ker m,m)$ determines and is determined by certain positive
integers $(n_1,n_2,\dots,n_{s})$ with $n_{i+1}|n_{i}$ for any $1\leq s-1$ and $n_1n_2\cdots n_{s}\mid n$. We show
that these integers also determine the conjugacy class of $F$ if it satisfies the condition $(*)$. For a closed
abelian subgroup $F$ of $\O(n)/\langle-I\rangle$ (or $\Sp(n)/\langle-I\rangle$), similarly we define an antisymmetric
bimultiplicative function $m: F\times F\longrightarrow\{\pm{1}\}$. With this, we have a subgroup $\ker m$ of $F$.
Moreover we define a subgroup $B_{F}$ of $\ker m$, which is always a diagonalizable subgroup. Using it, $F$ is
expressed in a blockwise form. We are able to describe the image of the projection of $F$ to each component if
it satisfies the condition $(*)$; and we describe the conjugacy class of $F$ in terms of the integer
$k=\frac{1}{2}\rank(F/\ker m)$, the conjugacy class of $B_{F}$ and a combinatorial datum arising from $m$ and
the images of projections of $F$ to block components. Such combinatorial data are generalization of symplective
metric spaces studied in \cite{Yu}, which arose from the study of elementary abelian $2$-subgroups. Given a closed
abelian subgroup $F$ of $\PU(n)\rtimes\langle\tau_0\rangle$ not contained in $\PU(n)$, we show that $F$ lifts to
a closed abelian subgroup $F'$ of $(\U(n)/\langle-I\rangle)\rtimes\langle\tau\rangle$ with the conjugacy class of
$F$ and the conjugacy class of $F'$ determine each other. Then, the classification is similar to the classification
of closed abelian subgroups of $\O(n)/\langle-I\rangle$ satisfying the condition $(\ast)$.


Given a compact simple Lie algebra $\frg_0$ with a complexification $\frg=\frg_0\otimes_{\bbR}\bbC$,
conjugacy classes of compact subgroups of $\Aut(\frg)$ and that of $\Aut(\frg_0)$ are in one-to-one
correspondence (cf. \cite{An-Yu-Yu}, Section 8). On the other hand, conjugacy classes of maximal
compact abelian subgroups of $\Aut(\frg)$ are in one-to-one correspondence with isomorphism classes of
fine group gradings of $\frg$ (cf. \cite{Elduque-Kochetov2}, Section 2). The study of group gradings was
initiated in \cite{PZ}. The fine group gradings of classical simple Lie algebras are classified in
\cite{Bahturin-Kochetov} and \cite{Elduque}. Our approach through the study of maximal abelian subgroups
and the method employed here have
several advantages. First the treatment is uniform, which enables us to show relations between abelian subgroups
of different groups and similarity between the shape of the sets of abelian subgroups of them. Second the
fusion of abelian subgroups are clearly understood from our classification, with which we are able to describe
the Weyl groups $W(F)=N_{G}(F)/C_{G}(F)$ of maximal abelian subgroups $F$. In turn, our classification gives
a better understanding of maximal abelian subgroups and the associated group gradings.

\smallskip

\noindent{\it Notation and conventions.} Given a Lie group $G$, write $Z(G)$ for the center of $G$ and
$G_0$ for the neutral component of $G$. Given a subgroup $H$ of $G$, let $C_{G}(H)$ denote the centralizer
of $H$ in $G$ and $N_{G}(H)$ denote the normalizer of $H$ in $G$. For a subset $X$ of $G$, let
$\langle X\rangle$ denote the subgroup of $G$ generated by elements in $X$. For a quotient group $G=H/N$, let
$[x]=xN$ ($x\in H$) denote a coset. For a compact semisimple real Lie algebra $\frg_0$, let $\Aut(\frg_0)$ be the
group of automorphisms of $\frg_0$ and $\Int(\frg_0)=\Aut(\frg_0)_0$ be the group of inner automorphisms.
Denote by $\Z_{m}=\{\lambda I_{m}:\ |\lambda|=1\}$, which is the center of the unitary group $\U(n)$.
Let $I_{n}$ be the $n\times n$ identity matrix. We define the following matrices,
\[I_{p,q}=\left(\begin{array}{cc} -I_{p}&0\\ 0&I_{q}\\\end{array}\right),\quad
J_{n}=\left(\begin{array}{cc} 0&I_{n}\\ -I_{n}&0\\\end{array} \right),\quad
\J'_{n}=\left(\begin{array}{cc}0&\I_{n}\\\I_{n}&0\\\end{array}\right),\]
$$K_{n}=\left(\begin{array}{cccc} 0&0&0&I_{n}\\0&0&-I_{n}&0\\ 0&I_{n}&0&0\\-I_{n}&0&0&0\\\end{array}\right).$$

\smallskip



\section{Projective unitary groups}\label{SS:A}

Let $\bbR$, $\bbC$, $\bbH$ be the set of real numbers, complex numbers and quaternion numbers respectively, which
is either a field or a division ring. For $F$=$\bbR$, $\bbC$ or $\bbH$, let $\M_{n}(F)$ be the set of $n\times n$
matrices with entries in $F$. Let \begin{eqnarray*} &&\O(n)=\{X\in \M_{n}(\bbR): XX^{t}=\I\},\quad
\SO(n)=\{X\in \O(n): \det X=1\},\\&& \U(n)=\{X\in \M_{n}(\bbC): XX^{\ast}=\I\},\quad
\SU(n)=\{X\in \U(n): \det X=1\}, \\&& \Sp(n)=\{X\in \M_{n}(\bbH): XX^{\ast}=\I\}.\end{eqnarray*} Here $X^{t}$
denotes the transposition of a matrix $X$ and $X^{\ast}$ denotes the conjugate transposition of $X$. Defined as
sets in this way, $\O(n)$, $\SO(n)$, $\U(n)$, $\SU(n)$, $\Sp(n)$ are actually Lie groups, i.e., groups with a
smooth manifold structure. Moreover, they are compact Lie groups, i.e., the underlying manifolds are compact.
Also let $\PO(n)$, $\PSO(n)$, $\PU(n)$, $\PSU(n)$, $\PSp(n)$ be the quotients of the groups $\O(n)$, $\SO(n)$,
$\U(n)$, $\SU(n)$, $\Sp(n)$ modulo their centers. Let \begin{eqnarray*}&&\mathfrak{so}(n)=
\{X\in M_{n}(\bbR): X+X^{t}=0\},\\&& \mathfrak{su}(n)=\{X\in M_{n}(\bbC): X+X^{\ast}=0, \tr X=0\},\\&&
\mathfrak{sp}(n)=\{X\in M_{n}(\bbH): X+X^{\ast}=0\},\end{eqnarray*} Then, $\mathfrak{so}(n)$, $\mathfrak{su}(n)$,
$\mathfrak{sp}(n)$ are Lie algebras of $\SO(n)$, $\SU(n)$, $\Sp(n)$ respectively. They represent all isomorphism
classes of compact classical simple Lie algebras.

\smallskip

Let $G=\PU(n)=\U(n)/\Z_{n}$, the projective unitary group of degree $n$. Let $F$ be a closed abelian subgroup of
$G$. For any $x,y\in F$, choose $A,B \in\U(n)$ representing $x,y$. That is, $x=[A]=A\Z_{n}$ and $y=[B]=B\Z_{n}$.
Since $1=[x,y]=[A,B]\Z_{n}$, $[A,B]=\lambda_{A,B} I$ for a complex number $\lambda_{A,B}$ with $|\lambda_{A,B}|=1$.
It is clear that the number $\lambda_{A,B}$ depends only on $x,y$, not on the choice of $A$ and $B$. By this, we
define a map $m: F\times F\longrightarrow\U(1)$ by $m(x,y)=\lambda_{A,B}$. Since
$$1=\det ABA^{-1}B^{-1}=\det(\lambda_{A,B}I)=(\lambda_{A,B})^{n},$$ $m(x,y)=\lambda_{A,B}=e^{\frac{2k\pi i}{n}}$ for
some integer $k$. The conclusion of the following lemma is clear.

\begin{lemma}\label{L:A-m}
The function $m$ is antisymmetric and bimultiplicative. That means, $m(x,x)=1$, $m(x,y)=m(y,x)^{-1}$ and
$m(xy,z)=m(x,z)m(y,z)$ for any $x,y,z\in F$.
\end{lemma}



Let $\ker m=\{x\in F: m(x,y)=1,\forall y\in F\}$. It is a subgroup of $F$ and the induced antisymmetric
bimultiplicative function $m$ on $F/\ker m$ is nondegenerate.

\begin{lemma}\label{L:A-kerm}
If $F$ is a closed abelian subgroup of $\PU(n)$ satisfying the condition $(\ast)$, then $\ker m=F_0$.
\end{lemma}

\begin{proof}
Since $m$ is a continuous map with finite image, one has $F_0\subset\ker m$. For any $x\in\ker m$,
substituting $F$ by a subgroup conjugate to it if necessary, we may assume that $x=A\Z_{n}$ for some
$A=\diag\{\lambda_1 I_{n_1},\lambda_2 I_{n_2},\dots,\lambda_{s}I_{n_{s}}\},$ where $n_1,\dots,n_{s}\in\bbZ$,
$n_1+\cdots+n_{s}=n$, and $\lambda_1,\dots,\lambda_{s}$ are distinct nonzero complex numbers. Since
$x\in\ker m$, one has $F\subset \U(n)^{A}/\Z_{n}$. From the condition of $\dim\frg_0^{F}=\dim F$, one gets
$Z(\U(n)^{A}/\Z_{n})_0\subset F_0$. Since $\U(n)^{A}=\U(n_1)\times\cdots\times\U(n_{s})$ and
$Z(\U(n)^{A}/\Z_{n})_0=(\Z_{n_1}\times\cdots\times\Z_{n_2})/\Z_{n}$, we get
$x\in Z(\U(n)^{A}/\Z_{n})_0\subset F_0$.
\end{proof}

Given a positive integer $k$ and a multiple $n=mk$ of $k$, define a subgroup $H_{k}$ of $\PU(n)$ by \[H_{k}
=\langle[\diag\{I_{m},\omega_{k}I_{m},\cdots,\omega_{k}^{k-1}I_{m}\}],[\left(\begin{array}{ccccc}0&I_{m}&
\cdots&0\\0&0&\ddots&\ddots\\\vdots&\ddots&\ddots&I_{m}\\I_{m}&0&\cdots&0\end{array}\right)]\rangle.\]

\begin{prop}\label{P:A}
For a closed abelian subgroup $F$ of $G$ satisfying the condition $(*)$, there exists positive integers
$n_1\geq n_2\geq\dots\geq n_{s}\geq 2$ with $n_{i+1}|n_{i}$ for any $1\leq i\leq s-1$ and
$n_1n_2\cdots n_{s}|n$ such that $F$ admits a decomposition $F=H_{n_1}\times \cdots H_{n_{s}}\times T$,
with $T$ a torus of dimension $m-1$, where $m=\frac{n}{n_1n_2\cdots n_{s}}$. Moreover, the conjugacy class
of $F$ is uniquely determined by the positive integers $(n_1,\dots,n_{s})$.
\end{prop}

\begin{proof}
We prove it by induction on the order of $F/F_0$. If $|F/F_0|=1$, i.e., $F$ is connected, it must be a maximal
torus of $G$, which corresponds to the case of $s=0$ and $m=n$ in the conclusion. In general, choose $x_1,y_1
\in F$ such that $m(x_1,y_1)$ is of maximal order. Let $n_1=o(m(x_1,y_1))$ and \[F_1=\{x\in F: m(x,x_1)=
m(x,y_1)=1\}.\] We first show that: for any $x,y\in F$, $m(x,y)^{n_1}=1$; there exists $x'_1\in x_1F_0$ and
$y'_1\in y_1F_0$ such that $o(x'_1)=o(y'_1)=n_1$; for any of such $x'_1,y'_1$, $F$ admits a decomposition
$F=\langle x'_1,y'_1\rangle\times F_1$. Let $z=m(x_1,y_1)$. Since $m$ is bimultiplicative and $m(x_1,y_1)$ is
of maximal order, $\{m(x_1,\cdot)\}=\{m(\cdot,y_1)\}=\langle z\rangle.$ Then, for any other $x\in F$, there
exists integers $a,b$ such that $m(x_1,x)=z^{a}$ and $m(x,y_1)=z^{b}$. Hence $x'=xx_1^{-b}y_1^{-a}\in F_1$.
Therefore $F=\langle x_1,y_1\rangle\cdot F_1$. Suppose there exists $x,y\in F$ such that $m(x,y)^{n_1}\neq 1$.
Let $x=x_1^{a_1}y_1^{b_1}x_2$ and $y=x_1^{a_2}y_1^{b_2}y_2$ for some $a_1,a_2,b_1,b_2\in\bbZ$ and $x_2,y_2
\in F_1$. Then $m(x,y)=m(x_1,y_1)^{a_1b_2-a_2b_1}m(x_2,y_2)$. Hence $m(x_2,y_2)^{n_1}\neq 1$. Let $z'=
m(x_2,y_2)$. Thus \[m(x_1x_2,y_1^{a}y_2^{b})=m(x_1,y_1)^{a}m(x_2,y_2)^{b}=z^{a}z'^{b}.\] Since $z'^{n_1}
\neq 1$, we can choose $a,b\in\bbZ$ such that $o(z^{a}z'^{b})>n_1$, which contradicts the assumption that
$m(x_1,y_1)$ is of maximal order. Hence $m(x,y)^{n_1}=1$ for any $x,y\in F$. For any $y\in F$, we have
$m(x_1^{n_1},y)=m(x_1,y)^{n_1}=1$ by the first statement proved above. Then, $x_1^{n_1}\in\ker m=F_0$ by Lemma
\ref{L:A-kerm}. As $F_0$ is a torus, there exists $x'_1\in x_1 F_0$ with $(x'_1)^{n_1}=1$. Similarly there
exists $y'_1\in y_1 F_0$ with $(y'_1)^{n_1}=1$. On the other hand, one has $n_1|o(x'_1)$ and $n_1|o(y'_1)$ since
$m(x'_1,y'_1)=z$ has order $n_1$. For any of such $x'_1,y'_1$, $F=\langle x'_1,y'_1\rangle\cdot F_1$ is proved
above. If $x_{1}'^{a} y_{1}'^{b}\in F_1$ for some $a,b\in\bbZ$, then $1=m(x_{1}'^{a} y_{1}'^{b},y'_1)=z^{a}$.
Hence $n_1|a$. Similarly we have $n_2|b$. Therefore $x_{1}'^{a}y_1'^{b}=1$. By this we get $F=\langle x'_1,y'_1
\rangle\times F_1$. With this, we may assume that $o(x_{1})=o(y_{1})=o(x_{1}F_1)=o(y_{1}F_1)=n_1$ and
$F=\langle x_1,y_1\rangle\times F_1$. We may assume that $[x_1,y_1]=(\omega_{n_1})^{-1}$. Then, one can show
that $(x_1,y_1)\sim([A_1],[B_1])$, where \[A_{n_1}=\diag\{I_{m_1},\omega_{n_1}I_{m_1},\dots,
\omega_{n_1}^{n_1-1}I_{m_1}\}\] and \[B_{n_1}=\left(\begin{array}{ccccc}0_{m_1}&I_{m_1}&0_{m_1}&\cdots&0_{m_1}
\\0_{m_1}&0_{m_1}&I_{m_1}&\cdots&0_{m_1}\\&&&\vdots&\\0_{m_1}&0_{m_1}&0_{m_1}&\cdots&I_{m_1}\\I_{m_1}&0_{m_1}
&0_{m_1}&\cdots&0_{m_1}\\\end{array}\right),\] $m_1=\frac{n}{n_1}$. Then, $F_1\subset(\U(n)^{\langle A_1,B_1
\rangle})/\Z_{n}\cong\U(m_1)/\Z_{m_1}$ and $|F_1/(F_1)_0|=\frac{|F/F_0|}{n_{1}^{2}}<|F/F_0|$. By induction we
finish the proof. One has $n_2|n_1$ since $m(x,y)^{n_1}=1$ for any $x,y\in F$. Similarly we have $n_{i+1}|n_{i}$
for any $1\leq i\leq s-1$.
\end{proof}

Given a sequence $\vec{a}=(n_1,\dots,n_{s})$ with $n_{i+1}|n_{i}$ for any $1\leq i\leq s-1$, let $V_{\vec{a}}=
H_{n_1}\times\cdots\times H_{n_{s}}$ be an abelian subgroup with an antisymmetric bimultiplicative function
induced from the functions on $\{H_{n_{j}}: 1\leq s\}$ such that $H_{n_{i}}, H_{n_{j}}$ are orthogonal for any
$i\neq j$. Let $\Sp(V_{\vec{a}})$ be the group of automorphisms of $V_{\vec{a}}$ preserving the function on it.
Denote by $$F_{\vec{a},m}=H_{n_1}\times\cdots\times H_{n_{s}}\times T_{m},$$ where $H_{n_{i}}\subset\PU(n_{i})$
as defined above and $T_{m}$ is a maximal torus of $\PU(m)$, $n=n_1\times\cdots\times n_{s}\times m$, also
we regard $\U(n_1)\times\cdots\U(n_{s})\times\U(m)$ as a subgroup of $\U(n)$ by tensor product. By Proposition
\ref{P:A}, any closed abelian subgroup of $\PU(n)$ is conjugate to some $F_{\vec{a},m}$. Apparently
$F_{\vec{a},m}$ is a maximal abelian subgroup of $\PU(n)$.


\begin{prop}\label{P:Auto-A}
One has $W(F_{\vec{a},m})=\Hom(V_{\vec{a}},\U(1)^{m}/\Z_{m})\rtimes(S_{m}\times\Sp(V_{\vec{a}}))$.
\end{prop}

\begin{proof}
Let $F=F_{\vec{a},m}$. Since $F_0$ is stable under the action of $W(F)$ and the bimultiplicative function on
$F/F_0$ is also preserved, we have a homomorphism $p: W(F)\longrightarrow\Aut(F_0)\times\Sp(V_{\vec{a}})$.
Considering preservation of eigenvalues, one can show that $\Im p\subset S_{m}\times\Sp(V_{\vec{a}})$
and $\ker p=\Hom(V_{\vec{a}},\U(1)^{m}/\Z_{m})$. Using $F=F'\times T_{m}$ with $F'$ is a finite abelian
subgroup of $\PU(n/m)$ and $T_{m}$ is a maximal torus of $\PU(m)$, one sees that $\Im p\supset S_{m}\times
\Sp(V_{\vec{a}}).$ By this we reach the conclusion of the Proposition.
\end{proof}


By Proposition \ref{P:A}, given two positive integers $n,m$ with $m\mid n$, the group $\SU(n)/\langle\omega_{m}I
\rangle$ possesses a closed abelian subgroup satisfying the condition $(*)$ if and only if $n\mid m^{k}$ for
some $k\geq 1$.

\section{Projective orthogonal groups}\label{S:BD}

Let $G=\O(n)/\langle-I\rangle$ and $\pi:\O(n)\longrightarrow\O(n)/\langle-I\rangle$ be the natural projection.
Given an abelian subgroup $F$ of $G$, for any $x,y\in F$, choose $A,B\in\O(n)$ representing $x,y$. Since
$1=[x,y]=\pi([A,B])$, $[A,B]=\lambda_{A,B}I$ for some $\lambda_{A,B}=\pm{1}$. It is clear that
$\lambda_{A,B}$ depends only on $x,y$. We define a map $m: F\times F\longrightarrow\{\pm{1}\}$ by
$m(x,y)=\lambda_{A,B}$. Then, $m$ is an antisymmetric bimultiplicative function on $F$. Let
$\ker m=\{x\in F:m(x,y)=1,\forall y\in F\}$. As $m$ is a continuous homomorphism taking values in $\pm{1}$,
one has $F_0\subset\ker m$.

\begin{lemma}\label{L:BD-kerm}
If $F$ is a closed abelian subgroup of $G$ satisfying the condition $(\ast)$, then for any $x\in\ker m$, there
exists $y\in F_0$ such that $xy\sim [I_{p,n-p}]$ for some integer $p$.
\end{lemma}

\begin{proof}
Without loss of generality, we may assume that $x=[A]$ for some \[A=\diag\{-I_{p},I_{q}, A_{n_1}(\theta_1),
\dots,A_{n_{s}}(\theta_{s})\},\] where $0<\theta_1<\cdots<\theta_{s}<\pi$, $p+q+2(n_1+\cdots+n_{s})=n$,
$$A_{k}(\theta)=\left(\begin{array}{cc}\cos(\theta) I_{k}&\sin(\theta) I_{k}\\-\sin(\theta) I_{k}&
\cos(\theta) I_{k}\\\end{array}\right).$$ Since $x\in\ker m$, we have \[F\subset\O(n)^{A}/\langle-I\rangle=
(\O(p)\times\O(q)\times\U(n_1)\times\cdots\U(n_{s}))/\langle-I\rangle.\] From $\dim\frg_0^{F}=\dim F$, we get
$\pi(1\times 1\times\Z_{n_1}\times\cdots\times\Z_{n_{s}})=Z(\O(n)^{A}/\langle-I\rangle)_0\subset F_0$. Let
$y=[\diag\{I_{p},I_{q}, A_{n_1}(-\theta_1),\dots,A_{n_{s}}(-\theta_{s})\}]$. Then $y\in F_0$ and
$xy=[I_{p,n-p}]$.
\end{proof}

Let \[B_{F}=\{x\in\ker m: A^{2}=1,\forall A\in\pi^{-1}(x)\}.\] 
For an even integer $n$, let \[H_2=\langle[I_{n/2,n/2}],[J'_{n/2}]\rangle\subset\O(n)/\langle-I\rangle.\]
Then, \[(\O(n)/\langle-I\rangle)^{H_2}=\Delta(\O(n/2)/\langle-I\rangle)\times H_2,\] where $\Delta(A)=
\diag\{A,A\}\in\O(2m)$ for any $A\in\O(m)$. For an integer $n$ with $4\mid n$, let \[H'_2=\langle[J_{n/2}],
[K_{n/4}]\rangle\subset\O(n)/\langle-I\rangle.\] Then, \[(\O(n)/\langle-I\rangle)^{H'_2}=\phi(\Sp(n/4)/\langle
-I\rangle)\times H'_2,\] where \[\phi(A+\mathbf{i}B+\mathbf{j}C+\mathbf{k}D)=\left(\begin{array}{cccc}A&C&B&D\\
-C&A&D&-B\\-B&-D&A&C\\-D&B&-C&A\\\end{array}\right)\] for $A+\textbf{i}B+\textbf{j}C+\textbf{k}D\in\Sp(n/4)$.


\begin{lemma}\label{L:BD-1}
If $F$ is a closed abelian subgroup of $G$ satisfying the condition $(\ast)$ and with $B_{F}=1$, then there
exists $k\in\bbZ_{\geq 0}$ such that $F=(H_2)^{k}\times\Delta^{k}(F')$, where $n'=\frac{n}{2^{k}}=1$ or $2$,
$F'=\SO(2)/\{\pm{I}\}$ if $n'=2$, and $F'=1$ if $n'=1$.
\end{lemma}

\begin{proof}
Since $B_{F}=1$, by Lemma \ref{L:BD-kerm} one has $\ker m=F_0$. Then, the induced function $m$ on $F/F_0$ is
nondegenerate. Choose a finite subgroup $F'$ of $F$ such that $F=F'\times F_0$. Then the function $m$ on
$F'$ is nondegenerate and we have \[F_0\subset(\O(n)^{\pi^{-1}(F')})/\langle-I\rangle.\] Since
$x^{2}\in\ker m=F_0$ for any $x\in F$, $F'$ is an elementary abelian 2-group. Let $2k=\rank F'=\rank F/F_0$.
By \cite{Yu} Proposition 2.12, for a given integer $k$, the conjugacy class of $F'$ has two possibilities.
Precisely, let $\defe F'-1$ be the difference between the number of elements of $F'$ conjugate to
$\pi(I_{\frac{n}{2},\frac{n}{2}})$ and the number of elements of $F'$ conjugate to $\pi(J_{\frac{n}{2}})$. Then,
$\defe F'\neq 0$ and the conjugacy class of $F'$ is determined by $\rank F'$ and $\sign(\defe F')$.


If $\sign(\defe F')>0$, then $F'=(H_2)^{k}$ and \[\O(n)^{\pi^{-1}(F')}=\Delta^{k}(\O(n'))\times
\pi^{-1}(F'),\] where $n'=\frac{n}{2^{k}}$ and $\Delta^{k}=\Delta\circ\cdots\circ\Delta$ ($k$-times).
Since $F$ satisfies the condition $(\ast)$, $F_0$ is a maximal torus of $\O(n')/\langle-I\rangle$. Moreover,
we have $n'=1$ or $2$ since otherwise $F_0$ has an element $[A]$ ($A\in\O(n')$) with $A\neq\pm{I}$ and $A^2=I$,
which forces $B_{F}\neq 1$.

If $\sign(\defe F')<0$, then $F'=(H_2)^{k-1}\times H'_2$ and \[\O(n)^{\pi^{-1}(F')}=\Delta^{k-1}(\Sp(n'))\times
\pi^{-1}(F'),\] where $n'=\frac{n}{2^{k+1}}$, $\Delta^{k-1}=\Delta\circ\cdots\circ\Delta$ ($(k-1)$-times) and
$\Sp(n')\subset\O(4n')$ is an inclusion given by the map $\phi$ described ahead of this lemma. Since $F$
satisfies the condition $(\ast)$, $F_0$ is a maximal torus of $\Sp(n')/\langle-I\rangle$. We must have $n'=1$
since otherwise $B_{F}\neq 1$. On the other hand, when $n'=1$, using an element in $F_0$ conjugate to
$[\mathbf{i}]$, we can find another finite subgroup $F''$ of $F$ such that $F=F_0\times F''$ and with
$\defe F''>0$. Therefore we return to the case of $\sign(\defe F')>0$.
\end{proof}

\begin{prop}\label{P:BD-2}
If $F$ is a closed abelian subgroup of $G$ satisfying the condition $(\ast)$, then there exists $k\geq 0$ and
$s_0,s_1\geq 0$ such that $n=2^{k}s_0+2^{k+1}s_1$, $\dim F_0=s_1$, $\rank(F/\ker m)=2k$ and
$\rank(\ker m/F_0)\leq\max\{s_0-1,0\}$.
\end{prop}

\begin{proof}
Let $2k=\rank(F/\ker m)$. Since $B_{F}\subset\ker m$, one has $F\subset\O(n)^{\pi^{-1}(B_{F})}/\langle-I
\rangle$. By the definition of $B_{F}$, $\pi^{-1}(B_{F})$ is an abelian subgroup of $\O(n)$ with every
element conjugate to $I_{p,n-p}$ for some $p$, $0\leq p\leq n$. Hence $\pi^{-1}(B_{F})$ is a diagonalizable
subgroup of $\O(n)$. Without loss of generality, we may assume that \[\O(n)^{\pi^{-1}(B_{F})}=\O(n_1)\times
\cdots\times\O(n_{s})\] for some positive integers $n_1,\dots,n_{s}$ with $n_1+\cdots+n_{s}=n$. Let
$F_{i}\subset\O(n_{i})/\langle-I_{n_{i}}\rangle$ be the image of the projection of $F$ to $\O(n_{i})/\langle
-I_{n_{i}}\rangle$ and $p_{i}: F\longrightarrow F_{i}$ be the projection. Each $F_{i}$ as a subgroup of
$\O(n_{i})/\langle-I_{n_{i}}\rangle$ has a bimultiplicative function $m_{i}$ similar as the function $m$ on
$F$. An element $x\in F$ is of the form $x=[(A_1,\dots,A_{s})]$, $A_{i}\in\O(n_{i})$ for any $1\leq i\leq s$.
For any other $y=[(B_1,\dots,B_{s})]\in F$, by the definition of the  function $m$, one has $[A_{i},B_{i}]=
m(x,y)I_{n_{i}}$ for any $i$, $1\leq i\leq s$. Hence $m_{i}(p_{i}(x),p_{i}(y))=m(x,y)$ for any $x,y\in F$. We
prove that $B_{F_{i}}=1$ for any $i$. Suppose this fails. Then $F$ has an element $x=[(A_1,\dots,A_{s})]$ with
$o(A_{i})=2$, $A_{i}\neq\pm{I}$ and $x_{i}=\pi(A_{i})\in\ker m_{i}$. Since $x_{i}=\pi(A_{i})\in\ker m_{i}$, by
the equality $m_{j}(p_{j}(x),p_{j}(y))=m(x,y)$ for any $y\in F$ and any $1\leq j\leq s$, we get $x\in\ker m$.
By the proof of Lemma \ref{L:BD-kerm}, there exists $y=[(B_1,\dots,B_{s})]\in F_0$ such that $B_{i}=I$ and
$(A_{j}B_{j})^2=I$ for any $1\leq j\leq s$. Hence $xy\in B_{F}$. On the other hand, since $B_{i}=I$ and
$A_{i}\neq\pm{I}$, \[\O(n)^{\pi^{-1}(xy)}\not\supset\O(n_1)\times\cdots\times\O(n_{s}).\] This contradicts
that $\O(n)^{\pi^{-1}(B_{F})}=\O(n_1)\times\cdots\times\O(n_{s})$.

For any $1\leq i\leq s$, since $B_{F_{i}}=1$, by Lemma \ref{L:BD-1}, one has $n_{i}=2^{k}n'_{i}$ with
$n'_{i}=1$ or $2$, and the conjugacy class of $F_{i}$ is uniquely determined by the number $k$. Let $s_0$
be the number of indices $i$ with $n'_{i}=1$ and $s_1$ be the number of indices $i$ with $n'_{i}=2$. Then,
$n=2^{k}s_0+2^{k+1}s_1$. Without loss of generality we may assume that $n_{i}=2^{k}$ if $1\leq i\leq s_0$, and
$n_{i}=2^{k+1}$ if $s_0+1\leq i\leq s$. Therefore \[F_0\subset\O(n)^{\pi^{-1}(B_{F})}/\langle-I\rangle=
(\O(n_1)\times\cdots\times\O(n_{s}))/\langle(-I,\dots,-I)\rangle\] is conjugate to $1^{s_0}\times\SO(2)^{s_1}$
if $s_0>0$, or to $\SO(2)^{s_1}/\langle(-I,\dots,-I)\rangle$ if $s_0=0$. Hence $\dim F=s_1$. Without loss of
generality we may assume that $F_0=1^{s_0}\times\SO(2)^{s_1}$ if $s_0>0$, and $F_0=\SO(2)^{s_1}/\langle
(-I,\dots,-I)\rangle$ if $s_0=0$. Hence $B_{F}\cap F_0=1^{s_0}\times\{\pm{1}\}^{s_1}$ if $s_0>0$, and
$\{B_{F}\cap F_0=\pm{1}\}^{s_1}/\langle(-I,\dots,-I)\rangle$ if $s_0=0$. Therefore $\rank(B_{F}\cap F_0)=
s_1-\delta_{s_0,0}$. Moreover \begin{eqnarray*}&&\rank(\ker m/F_0)\\&=&\rank(B_{F}F_0/F_0)\\&=&
\rank(B_{F}/B_{F}\cap F_0)\\&=&\rank B_{F}-\rank(B_{F}\cap F_0)\\&=&\rank B_{F}-(s_1-\delta_{s_0,0})\\&
\leq& s_0+s_1-1-(s_1-\delta_{s_0,0})\\&=&s_0-1+\delta_{s_0,0}\\&=&\max\{s_0-1,0\}.\end{eqnarray*}
\end{proof}

As in the above proof, let $F_{i}$ be the image of $F$ under the projection $p_{i}:\O(n)^{\pi^{-1}(B_{F})}
\longrightarrow\O(n_{i})/\langle-I\rangle$. For $1\leq i\leq s_0$, $F_{i}$ is an elementary abelian
$2$-subgroup, which means $A_{i}^{2}=\pm{I}$ if $A_{i}$ is the $i$-th component of $A$ for an element
$x=[A]\in F$. Define $\mu_{i}: F \longrightarrow\{\pm{1}\}$ by $A_{i}^{2}=\mu_{i}(x)I$. Since
$p_{i}(\ker m)=1$ if $1\leq i\leq s_0$, $\mu_{i}$ descends to a map $F/\ker m\longrightarrow\{\pm{1}\}$.
Moreover $p_{i}: F/\ker m\longrightarrow F_{i}$ is an isomorphism transferring $(m,\mu_{i})$ to
$(m_{i},\mu)$. In this way, we get a linear structure $(m,\mu_1,\cdots,\mu_{s_0})$ on $F/\ker m$. Note
that, each $\mu_{i}$ is compatible with $m$, i.e., we have $m(x,y)=\mu_{i}(x)\mu_{i}(y)\mu_{i}(xy)$ for
any $x,y\in F$; and as the proof of Lemma \ref{L:BD-1} shows, $(F/\ker m,m,\mu_{i})\cong V_{0,k;0,0}$ for
each $1\leq i\leq s_0$ (cf. \cite{Yu} Subsection 2.4).

\begin{definition}
A finite-dimensional vector space $V$ over the field $\mathbb{F}_2=\mathbb{Z}/2\mathbb{Z}$ is called
a multi-symplectic metric space if it is associated with a map $m: V\times V\longrightarrow\mathbb{F}_2$
and maps $$\mu_1,\mu_2,\dots,\mu_{s}: V\longrightarrow \mathbb{F}_2$$ such that: $m(x,x)=0$, $m(x,y)=m(y,x)$
and $m(x+y,z)=m(x,z)m(y,z)$ for any $x,y,z\in F$; and $m(x,y)=\mu_{i}(x)+\mu_{i}(y)+\mu_{i}(xy)$ for any
$x,y\in V$ and $1\leq i\leq s$. Two multi-symplectic metric spaces $(V,m,\mu_1,\dots,\mu_{s})$ and
$(V',m',\mu'_1,\dots,\mu'_{s})$ are called isomorphic if there exists a linear space isomorphism
$f: V\longrightarrow V'$ transferring $(V,m,\mu_1,\dots,\mu_{s})$ to $(V',m',\mu'_{i_1},\dots,\mu'_{i_s})$
for a permutation $i_1,i_2,\dots,i_{s}$ of $1,2,\dots,s$. Denote by $\Aut(V,m,\mu_1,\dots,\mu_{k})$ the
group of linear isomorphisms $f: V\longrightarrow V$ which is also an isomorphism as multi-symplectic
metric space.
\end{definition}

Note that, in the above definition the order among $\mu_1,\mu_2,\dots,\mu_{s}$ is overlooked.



\begin{prop}\label{P:BD-4}
The conjugacy class of a closed abelian subgroup $F$ of $G$ satisfying the condition $(\ast)$ is determined
by the integer $k=\frac{1}{2}\rank(F/\ker m)$, the conjugacy class of the subgroup $B_{F}$, and the linear
structure $(m,\mu_1,\cdots,\mu_{s_0})$ on $F/\ker m$.
\end{prop}

\begin{proof}
Without loss of generality we may assume that \[\O(n)^{\pi^{-1}(B_{F})}=\O(n_1)\times\cdots\times\O(n_{s})\]
for some positive integers $n_1,\dots,n_{s}$ with $n_1+\cdots+n_{s}=n$; moreover, $n_{i}=2^{k}$ if $1\leq i
\leq s_0$, and $n_{i}=2^{k+1}$ if $s_0+1\leq i\leq s$. By Lemma \ref{L:BD-1}, $(F_i)_0=1$ if $1\leq i\leq
s_0$, and $(F_i)_0\cong\SO(2)/\langle-I\rangle$ if $s_0+1\leq i\leq s$. We may assume that the subgroup
$B_{F}$ and the maps $\mu_{i}: F\longrightarrow\{\pm{1}\}$, $1\leq i\leq s_0$ are given. From these we
determine the subgroup $F$ up to conjugacy. The conjugacy class of each $F_{i}$ is uniquely determined.
The issue is how to match them. For an $x\in F$, let $x=[(A_1,\dots,A_{s_0},A_{s_0+1},\dots,A_{s})]$, where
$A_{i}\in\O(n_i)$ for any $1\leq i\leq s$. For $1\leq i\leq s_0$, the conjugacy of $A_{i}$ is determined by
$\mu_{i}(x)$. For $s_0+1\leq i\leq s$, as $F_{i}/(F_i)_0$ is an elementary abelian 2-group with a
nondegenerate bimultiplicative function $m_{i}$ and $(F_i)_0$ is a one-dimensional torus, we can modify
$A_{i}$ to make it conjugate to $I_{2^k,2^k}$ or $J_{2^{k}}$. Inductively, we construct elements
$x_1,\cdots,x_{2k}$ of $F$ generating $F/\ker m$ with $m(x_{j_1},x_{j_2})=1$ if and only if $\{j_1,j_2\}=
\{2j-1,2j\}$ for some $1\leq j\leq k$, such that the conjugacy class of the tuple $(x_1,\cdots,x_{2k})$ is
determined by $\mu_1,\mu_2,\dots,\mu_{s_0}$. Since $F=\langle B_{F}, F_0, x_1,\dots,x_{2k}\rangle$, the
conjugacy class of $F$ is determined accordingly.
\end{proof}


\begin{prop}\label{P:BD-3}
If $F$ is a closed abelian subgroup of $G$ satisfying the condition $(\ast)$, then $F$ is an elementary
abelian 2-subgroup if and only if $s_0=s$ and $\mu_{i}=\mu_{j}$ for any $1\leq i,j\leq s_0$. If $F$ is a
maximal abelian subgroup, then $\rank(\ker m/F_0)=\max\{s_0-1,0\}$; if $\rank(\ker m/F_0)=\max\{s_0-1,0\}$
and $(s_0,s_1)\neq(2,0)$, then $F$ is a maximal abelian subgroup.
\end{prop}

\begin{proof}
The first statement is clear. For the second statement, one has \[F\subset(\O(n)^{B_{F}}/\langle-I
\rangle)^{F_0}\cong(\O(2^{k})^{s_0}\times\U(2^{k})^{s_1})/\langle(-I,\dots,-I)\rangle.\] Hence
$\{\pm{I}\}^{s_0}\times(\Z_{2^{k}})^{s_1}/\langle(-I,\dots,-I)\rangle\subset C_{G}(F)$. If $F$ is a maximal
abelian subgroup, then $\{\pm{I}\}^{s_0}\times(\Z_{2^{k}})^{s_1}/\langle(-I,\dots,-I)\rangle\subset F$.
Furthermore one has \[\{\pm{I}\}^{s_0}\times(\Z_{2^{k}})^{s_1}/\langle(-I,\dots,-I)\rangle\subset\ker m.\]
Thus $\rank(\ker m/F_0)=\max\{s_0-1,0\}.$ Without loss of generality we may assume that
$\pi^{-1}(F_0)=1^{2^{k}s_0}\times\Delta^{k}(\SO(2))^{s_1}$. Then, \[C_{G}(F_0)=(\O(2^{k}s_0)\times
\U(2^{k})^{s_1})/\langle(-I,\dots,-I)\rangle.\] If $\rank(\ker m/F_0)=\max\{s_0-1,0\}$, we may and do assume
that $B_{F}=\{\pm{I_{2^{k}}}\}^{s_0}\times(\pm{I_{2^{k+1}}})^{s_1}$. While $(s_0,s_1)\neq(2,0)$ or $(0,1)$,
one has $n\neq 2^{k+1}$ and hence $I_{2^{k},n-2^{k}}\not\sim-I_{2^{k},n-2^{k}}$. Thefore \begin{eqnarray*}&&
C_{G}(\ker m)\\&=&(C_{G}(F_0))^{B_{F}}\\&=&((\O(2^{k}s_0)\times\U(2^{k})^{s_1})/\langle(-I,\dots,-I)
\rangle)^{B_{F}}\\&=&(\O(2^{k})^{s_0}\times\U(2^{k})^{s_1})\langle(-I,\dots,-I)\rangle.\end{eqnarray*}
Since each $F_i$ is a maximal abelian subgroup and the function $m$ on $F/\ker m$ is nondegenerate, one has
\begin{eqnarray*}&&C_{G}(F)\\&=&\big((\O(2^{k})^{s_0}\times\U(2^{k})^{s_1})/\langle(-I,\dots,-I)\rangle
\big)^{F}\\&=&\langle F,Z\big((\O(2^{k})^{s_0}\times\U(2^{k})^{s_1})/\langle(-I,\dots,-I)\rangle\big)\rangle
\\&=&\langle F,\big(\{\pm{I_{2^{k}}}\}^{s_0}\times(\Delta^{k}(\U(1)))^{s_1}\big)/\langle(-I,\dots,-I)\rangle
\rangle\\&=&\langle F,\ker m\rangle\\&=& F.\end{eqnarray*} Thus $F$ is a maximal abelian subgroup. If
$(s_0,s_1)=(0,1)$, $F$ is clearly a maximal abelian subgroup.
\end{proof}


When $(s_0,s_1)=(2,0)$, one can show that $F$ is a maximal abelian subgroup if and only if it is not an
elementary abelian 2-subgroup.

In the below we describe Weyl groups of maximal abelian subgroups of $G$.
Given a maximal abelian subgroup $F$ of $G$, by Proposition \ref{P:BD-2}, we associate a function $m: F
\times F\longrightarrow\{\pm{1}\}$, integers $k$, $s$, $s_0$, $s_1=s-s_0$ with $n=2^{k}s_0+2^{k+1}s_1$
and maps $\mu_1,\mu_2,\dots,\mu_{s_0}: F\longrightarrow\{\pm{1}\}$. Denote by $F_{k,m,\vec{\mu}}$ a
maximal abelian subgroup of $G$ like this, where $\vec{\mu}$ means the unordered tuple $(\mu_1,\mu_2,
\dots,\mu_{s_0})$. For a map $\mu: F\longrightarrow\{\pm{1}\}$, let $a_{\mu}$ be the number of indices
$i$, $1\leq i\leq s_0$ such that $\mu_{i}=\mu$. Denote by $S_{\vec{\mu}}=\prod_{\mu} S_{a_{\mu}}$.

\begin{prop}\label{P:BD-Weyl}
There is an exact sequence \begin{eqnarray*}1&\longrightarrow&\Hom(F/\ker m,B_{F})\rtimes S_{\vec{\mu}}
\longrightarrow W(F_{k,m,\vec{\mu}})\\&\longrightarrow&\Aut(F/\ker m,m,\vec{\mu})\times(\{\pm{1}\}^{s_1}
\rtimes S_{s_1})\\&\longrightarrow& 1.\end{eqnarray*}
\end{prop}

\begin{proof}
Let $F=F_{k,m,\vec{\mu}}$. The induced action of $W(F)$ on $F/\ker m$ preserves $m$ and $\vec{\mu}$. Hence
there is a homomorphism $p: W(F)\longrightarrow \Aut(F/\ker m, m,\vec{\mu})\times W(F_0),$  which is apparently
a surjective map. It is clear that $W(F_0)\cong\{\pm{1}\}^{s_1}\rtimes S_{s_1}$. There is another homomorphism
$p':\ker p\longrightarrow W(\ker m)$. The image $p'(w)$ for an element $w\in\ker p$ is determined by the action
of $w$ on the first $s_0$-components of $B_{F}$, which induces a permutation on the $s_0$ components isomorphic
to $\O(2^{k})$ of $\O(n)^{\pi^{-1}(B_{F})}$ and hence a permuatution of the indices $\{1,2,\dots,s_0\}$, denoted
by $\sigma$. Since $w$ acts trivially on $F/\ker m$, $\mu_{\sigma(i)}=\mu_{i}$ for each $1\leq i\leq s_0$.
Therefore $\Im p'\subset (\{\pm{1}\}^{s_1}\rtimes S_{s_1})\times\prod_{\mu} S_{a_{\mu}}$. It is clear that
$\ker p'\subset\Hom(F/\ker m,\ker m)$. Moreover considering the preservation of eigenvalues, one shows
$\ker p'\subset\Hom(F/\ker m,B_{F})$. From the description of $F$ as in the proof of Proposition \ref{P:BD-3},
one can show that $$\ker p=\Hom(F/\ker m,B_{F})\rtimes S_{\vec{\mu}}.$$
We reach the conclusion of the proposition.
\end{proof}


For a general closed abelian subgroup $F$ of $G$ satisfying the condition $(*)$, $C_{G}(F)$ is not necessarily
an abelian subgroup and the description of $W(F)$ is more complicated.




\smallskip

As an illustration of the classification of maximal abelian subgroups of $\O(n)/\langle-I\rangle$, we
classify multi-symplectic metric spaces while $s=2$ or $3$.

\begin{prop}\label{P:multiple SMS}
Given a vector space $V$ of dimension $2k$ over the field $\mathbb{F}_2$ with a nondegenerate bilinear
form $m$, for any $k\geq 1$, there exist two isomorphism classes of multi-symplectic metric spaces
$(V,m,\mu_1,\mu_2)$ such that $(V,m,\mu_{i})\cong V_{0,k;0,0}$, $i=1,2$; for any $k\geq 2$, there exist
four isomorphism classes of multi-symplectic metric spaces $(V,m,\mu_1,\mu_2,\mu_3)$ such that
$(V,m,\mu_{i})\cong V_{0,k;0,0}$, $i=1,2,3$.
\end{prop}

\begin{proof}
For $s=2$, we have two cases to consider according to $\mu_1=\mu_2$ or $\mu_1\neq\mu_2$. While $\mu_1=
\mu_2$, there exists a unique isomorphism class since $(V,m,\mu_{1})\cong V_{0,k;0,0}$. While $\mu_1\neq
\mu_2$, since $\mu_1,\mu_2$ are both compatible with $m$, $\mu_2\mu_1^{-1}: V\longrightarrow\mathbb{F}_2$
is a homomorphism. Let $V'=\ker(\mu_2\mu_1^{-1})$. Then $V'\subset V$ is a subspace of codimension one.
Thus $\rank(\ker(m|_{V'}))=1$ and $\defe(V',\mu_1)=\frac{1}{2}(\defe(V,\mu_1)+\defe(V,\mu_2))>0$. Hence
$(V',m|_{V'},\mu_1|_{V'})\cong V_{1,k-1;0,0}$ (cf. \cite{Yu} Proposition 2.29). Therefore the isomorphism
class of $(V,m,\mu_1,\mu_2)$ is determined uniquely.


For $s=3$, we have three cases to consider: $\mu_1=\mu_2=\mu_3$; $\mu_1=\mu_2\neq\mu_3$; $\mu_1\neq\mu_2,
\mu_3$ and $\mu_2\neq\mu_3$. While $\mu_1=\mu_2=\mu_3$, there is a unique isomorphism class since
$(V,m,\mu_{1})\cong V_{0,k;0,0}$. While $\mu_1=\mu_2\neq\mu_3$, similarly as the above proof for $s=2$ and
in the case of $\mu_1\neq\mu_2$, we get a unique isomorphism class. While $\mu_1\neq\mu_2,\mu_3$ and
$\mu_2\neq\mu_3$, let $V'=\ker(\mu_2\mu_1^{-1})\cap\ker(\mu_3\mu_1^{-1})$. Then $V'$ is a subspace of
codimension two of $V$. There are two cases according to $m|_{V'}$ is degenerate or not. In the case of
$m|_{V'}$ is nondegenerate, let $V''=\{x\in V: m(x,y)=0, \forall y\in V'\}$. Then $V=V'\oplus V''$ and
$m|_{V''}$ is also nondegenerate. Since $\mu_1|_{V'}=\mu_2|_{V'}=\mu_3|_{V'}$, one has $\defe(V'',\mu_1|_{V''})
=\defe(V'',\mu_2|_{V''})=\defe(V'',\mu_3|_{V''})$. If $\defe(V'',\mu_1|_{V''})<0$, then $\mu_1|_{V''}=
\mu_2|_{V''}=\mu_3|_{V''}$ and so $\mu_1=\mu_2=\mu_3$, which contradicts the assumption. If $\defe(V'',
\mu_1|_{V''})>0$, then $(\mu_1|_{V''},\mu_2|_{V''},\mu_3|_{V''})$ has only one possibility if the order among
them is overlooked. Hence we get one isomorphism type of $(V,m,\mu_1,\mu_2,\mu_3)$ in this case. In the case
of $m|_{V'}$ is degenerate, choose $V'''\subset V'$ such that $V'=V'''\oplus\ker(m|_{V'})$ and let
$V''=\{x\in V: m(x,y)=0, \forall y\in V'''\}$. Then, $V=V'''\oplus V''$, $m|_{V'''}$, $m|_{V''}$ are
nondegenerate, and $\dim V''=4$. One can show that $\defe V'''>0$ and $\defe (\mu_1|_{V''})>0$. Moreover we
can find generators $x_1,y_1,x_2,y_2$ of $V''$ such that $V''=\langle x_1,y_1\rangle\oplus\langle x_2,y_2
\rangle$ as a symplectic vector space, \[(\mu_1(x_1),\mu_1(x_2),\mu_1(y_1),\mu_1(y_2))=(1,1,1,1),\]
\[(\mu_2(x_1),\mu_2(x_2),\mu_2(y_1),\mu_2(y_2))=(1,1,1,-1),\] \[(\mu_3(x_1),\mu_3(x_2),\mu_3(y_1),\mu_3(y_2))
=(1,1,-1,1).\] Thus we get one isomorphism class of $(V,m,\mu_1,\mu_2,\mu_3)$ in this case.
\end{proof}



\section{Projective symplectic groups}\label{S:BC}

Let $G=\Sp(n)/\langle-I\rangle$ and $\pi:\Sp(n)\longrightarrow\Sp(n)/\langle-I\rangle$ be the natural
projection. The classification of abelian subgroups of $\Sp(n)/\langle-I\rangle$ satisfying the
condition $(\ast)$ is similar as the classification in the $\O(n)/\langle-I\rangle$ case. We give the
main steps below, but omit some proofs. First we define a map $m: F\times F\longrightarrow\{\pm{1}\}$.
For any $x=[A],y=[B]\in F$, if $[A,B]=\lambda_{A,B}I$, then $m(x,y)=\lambda_{A,B}$. Hence $m$ is an
antisymmetric bimultiplicative function on $F$. Let \[\ker m=\{x\in F|m(x,y)=1,\forall y\in F\}.\]

\begin{lemma}\label{L:C-kerm}
If $F$ is a closed abelian subgroup of $G$ satisfying the condition $(\ast)$, then for any $x\in\ker m$,
there exists $y\in F_0$ such that $xy\sim [I_{p,n-p}]$ for some $p$, $0\leq p\leq n$.
\end{lemma}

Let \[B_{F}=\{x\in\ker m: A^{2}=1,\forall A\in\pi^{-1}(x)\}.\] It is a subgroup of $\ker m$. By Lemma
\ref{L:C-kerm}, $\ker m=B_{F}F_0$. Let \[H_2=\langle[I_{n/2,n/2}],[J'_{n/2}]\rangle\] and \[H'_2=
\langle\mathbf{i} I,\mathbf{j} I\rangle.\]

\begin{lemma}\label{L:C-1}
If $F$ is a closed abelian subgroup of $G$ satisfying the condition $(\ast)$ and with $B_{F}=1$, then
either $n=1$ and $F$ is a maximal torus, or there exists $k\geq 1$ such that $F=(H_2)^{k-1}\times H'_2
\times\Delta^{k-1}(F')$, where $n'=\frac{n}{2^{k-1}}=1$ or $2$, $F'=\SO(2)/\{\pm{I}\}$ if $n'=2$, and
$F'=1$ if $n'=1$.
\end{lemma}

\begin{proof}
Let $2k=\rank F/\ker m$. By Lemma \ref{L:C-kerm}, $\ker m=B_{F}F_0=F_0$ as we suppose that $B_{F}=1$.
If $k=0$, then $F=F_0$ and it is a maximal torus of $G$. In this case $n=1$ since it is assumed that
$B_{F}=1$. If $k\geq 1$, choosing a complement $F'$ of $F_0$ in $F$, then $F_0$ is a maximal torus of
$(\Sp(n)/\langle-I\rangle)^{F'}.$ Since the bimultiplicative function $m$ is nondegenerate on
$F/\ker m=F/F_0$, it is nondegenerate on $F'$. By \cite{Yu} Proposition 2.16, we have $F'=(H_2)^{k-1}
\times H'_2$ or $F=(H_2)^{k}$. Similarly as the proof of Lemma \ref{L:BD-1}, in the second case, we can
show that $\dim F>0$. Choosing some $[A]\in F_0$ with $A\in\Sp(n)$ and $A^2=-I$, we may replace $F'$
by a different finite subgroup conjugate to $(H_2)^{k-1}\times H'_2$ and hence return to the first case.
In the first case, we have $(\Sp(n)/\langle-I\rangle)^{F'}\cong\O(n/2^{k-1})/\langle-I\rangle.$ Therefore
$n/2^{k-1}=1$ or $2$ since $F_0$ is a maximal torus of $\O(n/2^{k-1})/\langle-I\rangle$ and it is assumed
that $B_{F}=1$.
\end{proof}

\begin{prop}\label{P:C-2}
If $F$ is a closed abelian subgroup of $G$ satisfying the condition $(\ast)$, then either $F$ is a maximal
torus of $G$, or there exists $k\geq 1$, $s_0,s_1\geq 0$ such that $n=2^{k-1}s_0+2^{k}s_1$, $\dim F_0=s_1$,
$\rank F/\ker m=2k$ and $\rank(\ker m/F_0)\leq\max\{s_0-1,0\}$.
\end{prop}

The proof is similar as that of Proposition \ref{P:BD-2}. Note that we could regard the case of $F$ being
a maximal torus as the case of $k=0$, $s_0=0$ and $s_1=n$. By Proposition \ref{P:C-2}, any abelian subgroup
of $\Sp(n)$ satisfying the condition $(*)$ is a maximal torus; in particular $\Sp(n)$ has no finite abelian
subgroups satisfying the condition $(*)$.

\smallskip

Given an abelian subgroup $F$ of $G$ satisfying the condition $(\ast)$, the centralizer
$\Sp(n)^{\pi^{-1}(B_{F})}$ possesses a blockwise decomposition \[\Sp(n)^{\pi^{-1}(B_{F})}\sim\Sp(n_1)\times
\cdots\times\Sp(n_{s}).\] Let $F_{i}$ be the image of $F$ under the $i$-th projection
$p_{i}:\Sp(n)^{\pi^{-1}(B_{F})}\rightarrow\Sp(n_{i})/\langle-I\rangle.$ If $F$ is  not a maximal torus,
by Lemma \ref{L:C-1} and Proposition \ref{P:C-2}, we may assume that $n_{i}=2^{k-1}$ if $1\leq i\leq s_0$,
and $n_{i}=2^{k}$ if $s_0+1\leq i\leq s$. For $1\leq i\leq s_0$, define $\mu_{i}: F/\ker m\longrightarrow
\{\pm{1}\}$ by $\mu_{i}(x)=\mu(p_{i}(x))$, $x\in F$. As $p_{i}(\ker m)=1$, $\mu_{i}$ is well defined.
Moreover $p_{i}: F/\ker m\longrightarrow F_{i}$ is an isomorphism transferring $(m,\mu_{i})$ to $(m_{i},\mu)$.
In this way, we get a linear structure $(m,\mu_1,\cdots,\mu_{s_0})$ on $F/\ker m$. Note that, each $\mu_{i}$
is compatible with $m$, i.e., $m(x,y)=\mu_{i}(x)\mu_{i}(y)\mu_{i}(xy)$ for any $x,y\in F$. Moreover by Lemma
\ref{L:C-1}, one has $(F/\ker m,m,\mu_{i})\cong V_{0,k-1;0,1}$ (cf. \cite{Yu}, Subsection 2.4) for each
$1\leq i\leq s_0$.


\begin{prop}\label{P:C-4}
The conjugacy class of a closed abelian subgroup $F$ of $G$ satisfying the condition $(\ast)$ is determined
by the integer $k=\frac{1}{2}\rank(F/\ker m)$, the conjugacy class of the subgroup $B_{F}$, and the linear
structure $(m,\mu_1,\cdots,\mu_{s_0})$ on $F/\ker m$.
\end{prop}

The proof is along the same line as that of Proposition \ref{P:BD-4}.


\begin{prop}\label{P:C-3}
If $F$ is a closed abelian subgroup of $G$ satisfying the condition $(\ast)$, then it is an elementary abelian
2-subgroup if and only if $s_0=s$ and $\mu_{i}=\mu_{j}$ for any $1\leq i,j\leq s_0$. If $F$ is a maximal abelian
subgroup of $G$, then $\rank(\ker m/F_0)=\max\{s_0-1,0\}$; if $\rank(\ker m/F_0)=\max\{s_0-1,0\}$ and
$(s_0,s_1)\neq(2,0)$, then $F$ is a maximal abelian subgroup.
\end{prop}

The proof is similar as that of Proposition \ref{P:BD-3}.


\smallskip

The description of Weyl groups is also similar as orthogonal groups case. Given a maximal abelian subgroup
$F$ of $G$, by Proposition \ref{P:C-2}, we associate a function $m: F\times F\rightarrow\{\pm{1}\}$, integers
$k$, $s$, $s_0$, $s_1=s-s_0$ with $n=2^{k}s_0+2^{k+1}s_1$ and maps $\mu_1,\mu_2,\dots,\mu_{s_0}: F
\longrightarrow\{\pm{1}\}$. Denote by $F_{k,m,\vec{\mu}}$ a maximal abelian subgroup of $G$ like this, where
$\vec{\mu}$ means the unordered tuple $(\mu_1,\mu_2,\dots,\mu_{s_0})$. For a map $\mu: F\longrightarrow
\{\pm{1}\}$, let $a_{\mu}$ be the number of indices $i$, $1\leq i\leq s_0$ such that $\mu_{i}=\mu$. Denote by
$S_{\vec{\mu}}=\prod_{\mu} S_{a_{\mu}}$.

\begin{prop}\label{P:C-Weyl}
There is an exact sequence \begin{eqnarray*}1&\longrightarrow&\Hom(F/\ker m,B_{F})\rtimes S_{\vec{\mu}}
\longrightarrow W(F_{k,m,\vec{\mu}})\\&\longrightarrow&\Aut(F/\ker m,m,\vec{\mu})\times(\{\pm{1}\}^{s_1}
\rtimes S_{s_1})\\&\longrightarrow& 1.\end{eqnarray*}
\end{prop}

\section{Twisted projective unitary groups}

For an integer $n\geq 2$, let $G=\PU(n)\rtimes\langle\tau_0\rangle$, where $\tau_0=$ complex conjugation.
Then, $\tau_0^{2}=1$ and $\tau_0([A])\tau_0^{-1}=[\overline{A}]$ for any $A\in\U(n)$. One knows that
$\Aut(\mathfrak{su}(n))\cong G$ if $n\geq 3$. Define $\overline{G}=\U(n)\rtimes\langle\tau\rangle$,
$\tau^{2}=1$ and $\tau A\tau^{-1}=\overline{A}$ for any $A\in\U(n)$. Let $\pi: \overline{G}\longrightarrow G$
be the adjoint homomorphism. Then $\ker\pi=\Z_{n}$. Let $\tilde{G}=(\U(n)/\langle-I\rangle)\rtimes\langle\tau
\rangle$ be a group with $\tau^2=1$ and $\tau[A]\tau^{-1}=[\overline{A}]$ for any $A\in\U(n)$. Let
$\pi':\tilde{G}\longrightarrow G$ and $p:\overline{G}\longrightarrow\tilde{G}$ be the natural projections.
Then, $\pi'\circ p=\pi$.

Given an abelian subgroup $F$ of $G$, let $H_{F}=F\cap G_0$.  Define a map $m:H_{F}\times H_{F}\longrightarrow
\U(1)$ by $m(x,y)=\lambda$ if $x=\pi(A)$, $y=\pi(B)$, $A,B\in\U(n)$ and $[A,B]=\lambda I$. It is clear that $m$
is well defined and it is an antisymmetric bimultiplicative function on $H_{F}$. Let
\[\ker m=\{x\in H_{F}: m(x,y)=1,\forall y\in H_{F}\}.\]

\begin{lemma}\label{L:TA-m}
If $F$ is a closed abelian subgroup of $G$ not contained in $G_0$, then $m(x,y)=\pm{1}$ for any $x,y\in H_{F}$,
and $u^{2}\in\ker m$ for any $u\in F-F\cap G_0$.
\end{lemma}

\begin{proof}
For $x,y\in H_{F}$ and $u\in F-F\cap G_0$, let \[x=\pi(A),\quad y=\pi(B),\quad u=\pi(C),\] $A,B\in\U(n)$,
$C\in\tau\U(n)$. Since $F$ is abelian, \[CAC^{-1}=\lambda_1 A,\ CBC^{-1}=\lambda_2 B,\ ABA^{-1}B^{-1}=\lambda I\]
for some complex numbers $\lambda_1,\lambda_2,\lambda\in\U(1)$. One has \begin{eqnarray*}&&\lambda I=[A,B]=
[\lambda_1 A,\lambda_2 B]\\&=&[CAC^{-1},CBC^{-1}]=C[A,B]C^{-1}\\&=&C(\lambda I)C^{-1}=\lambda^{-1} I.
\end{eqnarray*} Hence $m(x,y)=\lambda=\pm{1}$. On the other hand, \begin{eqnarray*}&&C^2B(C^{2})^{-1}=
C(CBC^{-1})C^{-1}\\&=&C(\lambda_2 B)C^{-1}=C(\lambda_2 I)C^{-1}(CBC^{-1})\\&=&(\lambda_2^{-1}I)(\lambda_2 B)=B,
\end{eqnarray*} so $u^{2}\in\ker m$.
\end{proof}

Choose an $u\in F-F\cap G_0$ and let $u=\pi(C)$, $C\in\tau\U(n)$. For any $A\in\pi^{-1}(\ker m)$, since $F$ is
abelian, $[C,A]=\lambda I$ for some complex number $\lambda\in\U(1)$. As we assume that $A\in\pi^{-1}(\ker m)$,
$\lambda$ does not depend on the choice of $u$ and $C$. Let \[\nu: \pi^{-1}(\ker m)\longrightarrow\U(1)\] be
defined by $\nu(A)=\lambda$. It is a group homomorphism.

\begin{lemma}\label{L:TA-nu}
The map $\pi:\ker\nu/\langle-I\rangle\longrightarrow\ker m$ is an isomorphism.
\end{lemma}
\begin{proof}
For any $x\in\ker m$, choosing $A\in\pi^{-1}(x)$, then $[C,A]=\lambda^{2}I$ for some $\lambda\in\U(1)$. By this
one has $[C,\lambda A]=[C,\lambda I][C,A]=\lambda^{-2}\lambda^{2}I=I$. Thus $\lambda A\in\pi^{-1}(x)\cap\ker\nu$
and hence $\pi:\ker\nu/\langle-I\rangle\longrightarrow\ker m$ is surjective. On the other hand, if $\pi([A])=1$,
then $A=\lambda I$ for some $\lambda\in\U(1)$ and $[C,A]=I$. Thus $I=[C,A]=[C,\lambda I]=\lambda^{-2}I$.
Therefore $A=\pm{I}$ and hence $\pi:\ker\nu/\langle-I\rangle\longrightarrow\ker m$ is injective.
\end{proof}

Let $B_{F}=\{A\in\ker\nu: A^{2}=I\}$. Note that, here we define $B_{F}$ as a subgroup of $\overline{G}$, rather
than a subgroup of $G$. The following lemma indicates that $\ker m=\pi(B_{F})F_0$.

\begin{lemma}\label{L:TA-kerm}
If $F$ is a closed abelian subgroup of $G$ satisfying the condition $(\ast)$ and being not contained in $G_0$,
then for any $x\in\ker m$, there exists $y\in F_0$ such that $xy\in\pi(B_{F})$.
\end{lemma}

\begin{proof}
Choose an $A\in\ker\nu\cap\pi^{-1}(x)$. Since $A\in\ker\nu$, one has $[C,A]=I$. Hence $A$ and $\tau A\tau^{-1}=
\overline{A}$ are similar matrices. By this the multiplicity of an eigenvalue $\lambda$ of $A$ is equal to the
multiplicity of the eigenvalue $\overline{\lambda}$ of $A$. We may assume that \[A=\diag\{-I_{p},I_{q},
A_{n_1}(\theta_1),\dots,A_{n_{s}}(\theta_{s})\},\] where $0<\theta_1<\cdots<\theta_{s}<\pi$, $p+q+2(n_1+\cdots
+n_{s})=n$, \[A_{k}(\theta)=\left(\begin{array}{cc}\cos(\theta) I_{k}&\sin(\theta) I_{k}\\-\sin(\theta) I_{k}
&\cos(\theta) I_{k}\\\end{array}\right).\] Since $A\in\ker\nu$, one has \[\pi^{-1}(F)\subset\overline{G}^{A}=
(\U(p)\times\U(q)\times L_1\times\cdots\times L_{s})\rtimes\tau,\] where \[L_{i}=\{\left(\begin{array}{cc}X&Y
\\-Y&X\end{array}\right)\in\U(2n_{i}): X,Y\in\M(n_{i},\bbC)\},\] $1\leq i\leq s$. It holds that
$(\overline{G}^{A})_0\supset 1\times 1\times Z'_{n_1}\times\cdots\times Z'_{n_{s}}$, where \[Z'_{n_{i}}=
\{\left(\begin{array}{cc}\cos(\theta)I_{n_{i}}&\sin(\theta)I_{n_{i}}\\-\sin(\theta)I_{n_{i}}&
\cos(\theta)I_{n_{i}}\end{array}\right): 0\leq\theta\leq 2\pi\}.\] Since $F$ satisfies the condition $(\ast)$,
one has $\pi((\overline{G}^{A})_0)\subset F_0$. Let \[y=\pi(\diag\{I_{p},I_{q}, A_{n_1}(-\theta_1),\dots,
A_{n_{s}}(-\theta_{s})\}).\] Then $y\in F_0$ and $xy\in\pi(B_{F})$.
\end{proof}

The following lemma is well known, a proof of it could be found in \cite{Brocker-Dieck}, Page 177.

\begin{lemma}\label{L:MCC}
Let $G$ be a compact (not necessarily connected) Lie group and $x\in G$ be an element. If $T$ is a maximal
torus of $(G^{x})_0$, then any other element $y$ of $xG_0$ is conjugate to an element in $xT$ and any maximal
torus of $(G^{y})_0$ is conjugate to $T$.
\end{lemma}

For an even integer $n$, let $H_2=\langle[I_{n/2,n/2}],[J'_{n/2}]\rangle\subset\PU(n)$. Then,
\[((\U(n)/\Z_{n})\rtimes\langle\tau_0\rangle)^{H_2}=(\Delta(\U(n/2)/\Z_{n/2})\rtimes\langle\tau_0\rangle)
\times H_2,\] where $\Delta(A)=\diag\{A,A\}\in\U(2m)$ for any $A\in\U(m)$.

\begin{lemma}\label{L:TA-1}
Let $F$ be a closed abelian subgroup of $G$ satisfying the condition $(\ast)$ and being not contained in $G_0$.
If $B_{F}=\{\pm{I}\}$, then there exists $k\in\bbZ_{\geq 0}$ such that $F=(H_2)^{k}\times\Delta^{k}(F')$,
where $F'$ is a maximal torus of $(\U(n')/\Z_{n'})\rtimes \langle\tau_0\rangle$ and $n'=\frac{n}{2^{k}}=1$ or
$2$. 
\end{lemma}

\begin{proof}
Since $B_{F}=\{\pm{I}\}$, one has $\ker m=F_0$ by Lemma \ref{L:TA-kerm}. Then, the induced function $m$ on
$H_{F}/F_0$ is nondegenerate. Let $2k=\rank(H_{F}/\ker m)$. Choosing a finite subgroup $F'$ of $H_{F}$ such that
$H_{F}=F'\times F_0$, then the function $m$ on $F'$ is nondegenerate. Thus $F'$ is an elementary abelian 2-group
by Lemma \ref{L:TA-m}. By \cite{Yu}, Proposition 2.4, any non-identity element of $F'$ is conjugate to
$\pi(I_{\frac{n}{2},\frac{n}{2}})$ and the conjugacy class of $F'$ is uniquely determined by $k=\frac{1}{2}
\rank F'$. Hence $F'\sim (H_2)^{k}$. Substituting $F$ by a subgroup conjugate to it if necessary, we may assume
that \[C_{G}(F')=F'\times\Delta^{k}((\U(n')/\Z_{n'})\rtimes\langle\tau_0\rangle),\] where $n'=\frac{n}{2^{k}}$.
Then, there exists $\tau'\in F-H_{F}$ such that \[\langle F_0,\tau'\rangle\subset(\U(n')/\langle\Z_{n'}\rangle)
\rtimes\langle\tau_0\rangle.\] Since $F$ satisfies the condition $(\ast)$, $F_0$ is a maximal torus of
$(\U(n')/\Z_{n'})^{\tau'}.$ By Proposition \ref{L:MCC}, $F_0$ is conjugate to a maximal torus of $\SO(n')$ (if
$2\nmid n$) or $\SO(n')/\langle-I\rangle$ (if $2|n'$). As it is assumed that $B_{F}=\{\pm{I}\}$, one has $n'=1$
or $2$. We reach the conclusion of the proposition.
\end{proof}

\begin{lemma}\label{L:TA-2}
If $F$ is an abelian subgroup of $G$ not contained in $G_0$, then there exists an abelian subgroup $F'$ of
$\tilde{G}$ such that $\pi'(F')=F$ and $F'\cap(\Z_{n}/\langle-I\rangle)=\langle iI\rangle/\langle-I\rangle$.
Given $F$, $F'$ is determined up to conjugation by an element in $\Z_{n}/\langle-I\rangle$.
\end{lemma}

\begin{proof}
By Lemma \ref{L:TA-kerm}, $m(x,y)=\pm{1}$ for any $x,y\in H_{F}$. Hence $\pi'^{-1}(H_{F})$ is an abelian
subgroup of $\U(n)/\langle-I\rangle$. Choose any $u\in\pi'^{-1}(F-H_{F})$ and let\[F'=\langle
(\pi'^{-1}(H_{F}))^{u},u\rangle\subset(\U(n)/\langle-I\rangle)\rtimes\langle\tau\rangle.\] For any
$x'\in\pi'^{-1}(H_{F})$, since $F$ is abelian, one has $ux'u^{-1}x'^{-1}=[\lambda^{2}I]$ for some $\lambda\in
\U(1)$. Then $u(\lambda x')u^{-1}=\lambda x'$ and hence $\lambda x'\in F'$. Therefore $\pi(F')=F$. If
$[\lambda I]\in F'\cap(\Z_{n}/\langle-I\rangle)$, then $[\lambda I]=u[\lambda I]u^{-1}=[\lambda^{-1}I]$. Thus
$[\lambda I]=1$ or $[iI]$ and hence $F'\cap(\Z_{n}/\langle-I\rangle)=\langle iI\rangle/\langle-I\rangle$. For
a subgroup $F'$ of $(\U(n)/\langle-I\rangle)\rtimes\langle\tau\rangle$ satisfying the conditions in the
proposition, choose $x\in F-H_{F}$ and $u\in F'\cap \pi'^{-1}(x)$. Since $[\lambda I]u[\lambda I]^{-1}=
[\lambda^{2} I]u$, the conjugacy class of $u$ is determined up to conjugation by an element in
$\Z_{n}/\langle-I\rangle$. Fixing $u$, for any $y\in H_{F}$ and $y'\in\pi'^{-1}(y)$, if $y'$ and $[\lambda I]y'$
both commute with $u$, then \begin{eqnarray*}&&1\\&=&u([\lambda I]y')u^{-1}([\lambda I]y')^{-1}\\&=&
(u[\lambda I]u^{-1}[\lambda I]^{-1})[\lambda I](uy'u^{-1}y'^{-1})[\lambda I]^{-1}\\&=&[\lambda^{-2} I]
[\lambda I][\lambda I]^{-1}\\&=&[\lambda^{-2} I].\end{eqnarray*} Hence $[\lambda I]=1$ or $[iI]$. As
$[iI]\in F'$, $F'$ is determined by $u$. Therefore the conjugacy class of $F'$ is determined up to conjugation
by an element in $\Z_{n}/\langle-I\rangle$.
\end{proof}

In this correspondence, one can show that $F$ satisfies the condition $(\ast)$ if and only if $F'$ satisfies it;
$F$ is a maximal abelian subgroup if and only if $F'$ is. Denote by $H_{F'}=F'\cap\tilde{G}$.

Given a closed abelian subgroup $F$ of $G$ satisfying the condition $(\ast)$ and being not contained in
$G_0$, by Lemma \ref{L:TA-2}, $F$ lifts to an abelian subgroup $F'$ of $\tilde{G}=(\U(n)/\langle-I\rangle)
\rtimes\langle\tau\rangle$ satisfying the condition $(\ast)$ and being not contained in $\tilde{G}_0$. We
define an antisymmetric bimultiplicative function $m': H_{F'}\times H_{F'}\longrightarrow\{\pm{1}\}$ and
a homomorphism $\nu':\ker m'\longrightarrow\{\pm{1}\}$ by $[A,B]=m'(x,y)I$, $[C,A']=\nu'(x')I$ for
$x=[A],y=[B]\in H_{F'}$, $x'=[A']\in\ker m'$, $u=[C]\in F'-H_{F'}$. Note that $\nu'([iI])=-1$. Hence
$$\ker m'=\ker\nu'\times\langle[iI]\rangle.$$ Let
$B_{F'}=\{x\in\ker\nu': A^2=I, \forall A\in p^{-1}(x)\}$. Similarly as Lemma \ref{L:TA-kerm}, one can show
that $$\ker\nu'=B_{F'}F_0.$$ It is clear that $B_{F'}=p(B_{F})$. The following lemma is analogous to Lemma
\ref{L:TA-1}, which follows from Lemmas \ref{L:TA-1} and \ref{L:TA-2}.

\begin{lemma}\label{L:TA-3}
Let $F'$ be a closed abelian subgroup of $\tilde{G}$ satisfying the condition $(\ast)$ and being not contained
in $\tilde{G}_0$. If $B_{F'}=1$, then there exists $k\in\bbZ_{\geq 0}$ such that $F=(H_2)^{k}\times\langle[iI]
\rangle\times\Delta^{k}(F')$, where $n'=\frac{n}{2^{k}}=1$ or $2$, $F'=\SO(2)/\{\pm{I}\}\times\langle\tau_0
\rangle$ if $n'=2$, and $F'=\langle\tau_0\rangle$ if $n'=1$.
\end{lemma}

Without loss of generality, we may assume that \[(\overline{G})^{B_{F}}=(\U(n_1)\times\cdots\U(n_{s}))\rtimes
\langle\tau\rangle\] for some positive integers $n_1,\dots,n_{s}$ with $n_1+\cdots+n_{s}=n$. Let $F'_{i}$ be
the image of the projection of $F'$ to $(\U(n_{i})/\langle-I\rangle)\rtimes\langle\tau\rangle$ and
$p_{i}: F'\rightarrow(\U(n_{i})/\langle-I\rangle)\rtimes\langle\tau\rangle$ be the projection map. One can show
that $B_{F'_{i}}=1$ for each $i$, $1\leq i\leq s$. Let $k=\frac{1}{2}\rank(H_{F'}/\ker m')$. By Lemma
\ref{L:TA-3}, $n_{i}=2^{k}$ or $2^{k+1}$. We may assume that $n_1=n_2=\cdots=n_{s_0}=2^{k}$, $n_{s_0+1}=
\cdots=n_{s}=2^{k+1}$. For $1\leq i\leq s_0$, define $\mu_{i}: H_{F'}/\ker\nu'\longrightarrow\{\pm{1}\}$ by
$\mu_{i}(x)=\mu(p_{i}(x))$. That is, for $x=[(A_1,A_2,\dots,A_{s})]\in H_{F'}$, one has $A_{i}^{2}=\mu_{i}(x)I$
for $1\leq i\leq s_0$. As $p_{i}(\ker\nu')=1$, $\mu_{i}$ is well defined. Moreover $p_{i}: H_{F'}/\ker\nu'
\longrightarrow H_{F'_{i}}$ is an isomorphism transferring $(m,\mu_{i})$ to $(m_{i},\mu)$. In this way, we
define a linear structure $(m,\mu_1,\cdots,\mu_{s_0})$ on $H_{F'}/\ker\nu'$. Note that, each $\mu_{i}$ is
compatible with $m'$, i.e., \[m'(x,y)=\mu_{i}(x)\mu_{i}(y)\mu_{i}(xy)\] for any $x,y\in H_{F'}$. By Lemma
\ref{L:TA-3}, as a symplectic metric space one has (cf. \cite{Yu}, Subsection 2.4)
\[(H_{F'}/\ker\nu',m,''\mu_{i})\cong V_{0,k;1,0}\] for any $i$, $1\leq i\leq s_0$.

\begin{prop}\label{P:TA-5}
The conjugacy class of a closed abelian subgroup $F'$ of $\tilde{G}$ satisfying the condition $(\ast)$ is
determined by the integer $k=\frac{1}{2}\rank(H_{F'}/\ker m')$, the conjugacy class of the subgroup
$B_{F'}$, and the linear structure $(m',\mu_1,\cdots,\mu_{s_0})$ on $H_{F'}/\ker\nu'$.
\end{prop}

\begin{proof}
Without loss of generality we may assume that \[\overline{G}^{B_{F}}=(\U(n_1)\times\cdots\times\U(n_{s}))
\rtimes\langle\tau\rangle\] for some positive integers $n_1,\dots,n_{s}$ with $n_1+\cdots+n_{s}=n$; moreover,
$n_{i}=2^{k}$ if $1\leq i\leq s_0$, and $n_{i}=2^{k+1}$ if $s_0+1\leq i\leq s$, where $F=\pi'(F')$. By
Lemma \ref{L:TA-1}, $(F'_{i})_0=1$ if $1\leq i\leq s_0$, and $(F'_{i})_0\cong\SO(2)/\langle-I\rangle$ if
$s_0+1\leq i\leq s$. We may assume that the subgroup $B_{F'}$ and the maps $\mu_{i}: H_{F'}/\ker\nu'
\rightarrow\{\pm{1}\}$, $1\leq i\leq s_0$ are given. From these we determine the subgroup $F'$ up to conjuagacy.
The conjugacy class of each $F'_{i}$ is uniquely determined by Lemma \ref{L:TA-3}. The issue is how to match
them. For an $x\in H_{F'}$, let $x=[(A_1,\dots,A_{s_0},A_{s_0+1},\dots,A_{s})]$, where $A_{i}\in\U(n_i)$,
$1\leq i\leq s$, we have $A_{i}^2=\mu_{i}(x)I$ if $1\leq i\leq s_0$ and $A_{i}^{2}\in(F'_{i})_0$ if
$s_0+1\leq i\leq s$. For $1\leq i\leq s_0$, the conjugacy of $A_{i}$ is determined by $\mu_{i}(x)$. For
$s_0+1\leq i\leq s$, as $F'_{i}/(F'_i)_0$ is an elementary abelian 2-group with a nondegenerate bimultiplicative
function $m'_{i}$ and $(F'_i)_0$ is a one-dimensional torus, we can modify $A_{i}$ to make it conjugate to
$I_{2^k,2^k}$ or $J_{2^{k}}$. Inductively, we construct elements $x_1,\cdots,x_{2k}$ of $H_{F'}$ generating
$H_{F'}/\ker m'$ with $m(x_{j_1},x_{j_2})=1$ if and only if $\{j_1,j_2\}=\{2j-1,2j\}$ for some $1\leq j\leq k$,
such that the conjugacy class of the tuple $(x_1,\cdots,x_{2k})$ is determined by $\mu_1,\mu_2,\dots,\mu_{s_0}$.
In this way $H_{F'}=\langle F'_0,B_{F'},[iI],x_1,\dots,x_{2k}\rangle$ is determined accordingly. Since
$$p^{-1}(F')\subset \overline{G}^{B_{F}}=(\U(n_1)\times\cdots\times\U(n_{s}))\rtimes\langle\tau\rangle.$$
Fixing $H_{F'}$, the conjugacy class of an element $\eta\in F'-H_{F'}$ as an element of $p(\overline{G}^{B_{F}})$
is determined uniquely modulo $H_{F'}$. Therefore the conjugacy class of $F'$ is determined.
\end{proof}

\begin{remark}
From the above proof, one sees that $F$ always contains an automorphism of order 2 or 4. There should be a
criterion of when $F$ contains an outer involution in terms of $\vec{\mu}=(\mu_1,\cdots,\mu_{s_0})$.
\end{remark}

The following two propositions are analogues of Propositions \ref{P:BD-2} and \ref{P:BD-3}, which could also
be proved along the same line.

\begin{prop}\label{P:TA-2}
If $F$ is a closed abelian subgroup of $G$ satisfying the condition $(\ast)$ and being not contained
in $G_0$, then there exists integers $k\geq 0$ and $s_0,s_1\geq 0$ such that $n=2^{k}s_0+2^{k+1}s_1$,
$\dim F_0=s_1$, $\rank(H_{F}/\ker m)=2k$ and $\rank(\ker m/F_0)\leq\max\{s_0-1,0\}$.
\end{prop}


\begin{prop}\label{P:TA-3}
If $F$ is a closed abelian subgroup of $G$ satisfying the condition $(\ast)$, then it is an elementary
abelian 2-subgroup if and only if $s_0=s$ and $\mu_{i}=\mu_{j}$ for any $1\leq i,j\leq s_0$. If $F$ is a
maximal abelian subgroup, then $\rank(\ker m/F_0)=\max\{s_0-1,0\}$; if $\rank(\ker m/F_0)=\max\{s_0-1,0\}$
and $(s_0,s_1) \neq(2,0)$, then it is a maximal abelian subgroup.
\end{prop}

Given a maximal abelian subgroup $F'$ of $\tilde{G}$, by Proposition \ref{P:TA-5}, we associate a function
$m': H_{F'}\times H_{F'}\longrightarrow\{\pm{1}\}$, integers $k$, $s$, $s_0$, $s_1=s-s_0$ with
$n=2^{k}s_0+2^{k+1}s_1$ and maps $\mu_1,\mu_2,\dots,\mu_{s_0}: H_{F'}/\ker\nu'\longrightarrow\{\pm{1}\}$.
Denote by $F'_{k,m',\vec{\mu}}$ a maximal abelian subgroup of $\tilde{G}$ like this and $F_{k,m',\vec{\mu}}=
\pi'(F'_{k,m',\vec{\mu}})$ be the corresponding maximal abelian subgroup of $G$. where $\vec{\mu}$
means the unordered tuple $(\mu_1,\mu_2,\dots,\mu_{s_0})$. For a map $\mu:  H_{F'}\longrightarrow\{\pm{1}\}$, let
$a_{\mu}$ be the number of indices $i$, $1\leq i\leq s_0$ such that $\mu_{i}=\mu$. Denote by
$S_{\vec{\mu}}=\prod_{\mu} S_{a_{\mu}}$. The following proposition describes the Weyl groups of maximal abelian
subgroups of $G$.

\begin{prop}\label{P:TA-Weyl}
For Let $F'=F'_{k,m',\vec{\mu}}$ and $F=F_{k,m',\vec{\mu}}$, there is an exact sequence \begin{eqnarray*}1&
\longrightarrow&\Hom(F'/\ker m',B_{F'})\rtimes S_{\vec{\mu}}\longrightarrow W(F_{k,m',\vec{\mu}})\\&
\longrightarrow&\Aut(H_{F'}/\ker\nu',m',\vec{\mu})\times(\{\pm{1}\}^{s_1}\rtimes S_{s_1})\longrightarrow 1.
\end{eqnarray*}
\end{prop}

\begin{proof}
There is a natural homomorphism $\phi: W(F')\longrightarrow W(F)$. We show that $\phi$ is surjective and
$\ker\phi=\langle\Ad([\frac{1+i}{\sqrt{2}}I])\rangle$. For $g\in\tilde{G}_0$, suppose that $\pi'(g)
\in N_{G}(F)$. Then, $\pi'(gF'g^{-1})=\pi'(F')$. That means $\Ad(g)F'$ is another lift of $F$. By Lemma
\ref{L:TA-3}, one has $\Ad(g)F'=\Ad([\lambda I])F'$ for some $\lambda I\in\Z_{n}$. Thus $[\lambda I]^{-1}g
\in N_{\tilde{G}}(F')$. Therefore $\phi$ is surjective. For $g\in\tilde{G}_0$, suppose that $\Ad(g)\in W(F')$
and $\phi(\Ad(g))=1$. Then, $\pi'(g)\in C_{G}(F)=F$. Hence $g\in p(\Z_{n})F'$. We may assume that
$g=[\lambda I]$ for some $\lambda\in\U(1)$. Choosing an element $u\in F'-H_{F'}$, thus $[\lambda^2 I]=
gug^{-1}u^{-1}\in F'$. Since $F'\cap p(\Z_{n})=p(\langle iI\rangle)$, one has $\Ad(g)\in\langle
\Ad([\frac{1+i}{\sqrt{2}}I])\rangle$. By this we have $$W(F)=W(F')/\langle\Ad([\frac{1+i}{\sqrt{2}}I])
\rangle.$$

The induced action of an $w\in W(F')$ on $H_{F'}/\ker\nu'$ preserves $m'$ and $\vec{\mu}$. Hence there is a
homomorphism $p: W(F')\longrightarrow \Aut(H_{F'}/\ker\nu', m',\vec{\mu})\times W(F'_0),$ which is apparently
a surjective map. Clearly one has $W(F'_0)=\{\pm{1}\}^{s_1}\rtimes S_{s_1}$. There is another homomorphism
$p': \ker p\longrightarrow W(\ker\nu').$ The action of an $w\in\ker p$ on $B_{F'}$ induces a permutation on
the first $s_0$ components isomorphic to $\U(2^{k})$ of $\overline{G}^{B_{F}}$ and hence a permuatution of the
indices $\{1,2,\dots,s_0\}$, denoted by $\sigma$. Since $w$ acts trivially on $H_{F'}/\ker\nu'$, one has
$\mu_{\sigma(i)}=\mu_{i}$ for each $1\leq i\leq s_0$. Therefore $\Im p'\subset\prod_{\mu} S_{a_{\mu}}$.
Considering the preservation of eigenvalues, one shows that $\ker p'\subset\Hom(F'/\ker m',B_{F'})\times
\langle\Ad([\frac{1+i}{\sqrt{2}}I])\rangle$. Moreover one can show that
$$\ker p/\langle\Ad([\frac{1+i}{\sqrt{2}}I])\rangle=\Hom(F'/\ker m',B_{F'})\rtimes S_{\vec{\mu}}.$$
Therefore we get the exact sequence in the conclusion.
\end{proof}


Jun Yu \\ School of Mathematics, \\ Institute for Advanced Study, \\ Einstein Drive, Fuld Hall, \\
Princeton, NJ 08540, USA \\
email:junyu@math.ias.edu.

\end{document}